\documentclass[envcountsame]{llncs}

\usepackage{amsmath,amssymb,enumerate,gastex,mathrsfs,cite}
\usepackage[all]{xy}

\newcommand{\inv}{^{-1}}

\newcommand{\bigast}{\mathop{\hbox{\Large $\ast$}}}

\begin{document}

\title{Submonoids and rational subsets of groups with infinitely many ends}
\author{Markus Lohrey\inst{1} \and Benjamin Steinberg\inst{2,}\thanks{The authors would like to
acknowledge the support of DFG Mercator program.  The second author is also supported by an NSERC grant.}
\institute{Universit\"at
Leipzig, Institut f\"ur Informatik, Germany
 \and
School of Mathematics and Statistics,
Carleton University, ON, Canada\\
\email{lohrey@informatik.uni-leipzig.de,
bsteinbg@math.carleton.ca}}}

\maketitle

\begin{abstract}
In this paper we show that the membership problems for finitely generated
submonoids and for rational subsets are recursively equivalent for groups with two or more ends.
\end{abstract}

\section{Introduction}
Let $G$ be a finitely generated group with finite generating set $\Sigma$.
Put $\widetilde{\Sigma} = \Sigma\cup \Sigma\inv$ and let $\pi\colon \widetilde{\Sigma}^*
 \to G$ be the canonical projection from the free monoid on $\widetilde{\Sigma}$ to $G$.
The (uniform) submonoid membership problem for $G$ takes as
input words $w,w_1,\ldots, w_n\in \widetilde{\Sigma}^*$ and asks whether
$\pi(w)$ belongs to the submonoid generated by $\pi(w_1),\ldots, \pi(w_n)$.
Note that this problem generalizes the subgroup membership problem (or generalized word problem) for $G$.

The (uniform) rational subset membership problem for $G$ takes as input a word
$w\in  \widetilde{\Sigma}^*$ and a finite automaton $\mathscr A$ over
$\widetilde{\Sigma}^*$ and asks whether $\pi(w)\in \pi(L(\mathscr A))$.
Recall that a finite automaton $\mathscr A$ over
$\widetilde{\Sigma}$ is a tuple $\mathscr A = (Q,\widetilde{\Sigma}, \delta, q_0, F)$ where
$\delta \subseteq Q \times \widetilde{\Sigma} \times Q$,
$q_0 \in Q$  is the initial state and $F \subseteq Q$ is the set of final states.
The language $L(\mathscr A)$ recognized by $\mathscr A$ is the set of all finite
words from $\widetilde{\Sigma}^*$ that label a path from the initial state
$q_0$ to some state from $F$.
A set of the form $\pi(L(\mathscr A))$ is called a rational subset of $G$.

Study of rational subsets of groups began with the pioneering work of
Benois  on free groups~\cite{Benois69} and Eilenberg and Sch\"utz\-en\-ber\-ger on
abelian groups~\cite{EiSchu69}.  Rational subsets of groups were further
studied in~\cite{AniSeif75}, where it was
shown that a subgroup is a rational subset if and only if
it is finitely generated, from which one obtains an easy proof of Howson's
Theorem (in essence the same as Stallings' proof~\cite{Stal83}).  Recently,
there has been renewed interest in the subject of rational subsets of
groups~\cite{Gru99,Gilman96,KaSiSt06,LohSte08,Ned00,Rom99,LohSen08}, especially in connection to solving
equations over groups~\cite{DiGuHa05,DahGu09} and the isomorphism problem for toral relatively hyperbolic groups~\cite{DahGr08}.  On the
other hand, work of Stephen~\cite{Ste93} and Ivanov, Margolis and
Meakin~\cite{IvMaMe01} on the word problem for certain inverse monoids has
spurred some interest in the submonoid membership problem.  In particular,
the word problem for one-relator inverse monoids motivated the
paper~\cite{MaMeSu05}, where membership in positively generated submonoids
of certain groups given by monoid presentations is considered.  Recently,
the authors have shown that the submonoid membership
problem is undecidable in free metabelian groups of rank $2$~\cite{LohSt09}.

Trivially, decidability of the rational subset membership problem implies
decidability of the submonoid membership problem.  In previous
work~\cite{LohSte08}, the authors showed that these two problems are
recursively equivalent for graph groups (also known as right-angled Artin
groups or free partially commutative groups) and gave the precise class of
graph groups for which these problems are decidable.  The authors were also
able to prove that these problems are recursively equivalent for certain
amalgamated free products (including free products) and HNN extensions with finite edge groups.  This
led the authors to conjecture that the rational subset membership
and submonoid membership problems are recursively
equivalent for groups with two or more ends.  The main
result of this paper is to establish our conjecture.

Recall that if $\Gamma$ is a locally finite graph, then the space of ends of
$\Gamma$ is the projective limit $\varprojlim \pi_0(\Gamma\setminus F)$ where
$F$ runs over all finite subgraphs of $\Gamma$ and $\pi_0(X)$ is the set of
connected components of $X$.  If $G$ is a group with finite generating set
$\Sigma$, then the number of ends $e(G)$ of $G$ is the cardinality of the
space of ends of its Cayley graph $\Gamma$ with respect to $\Sigma$; it is
well known that this number depends only on $G$ and not $\Sigma$.  Moreover,
it is a result of Hopf that $e(G)$ is either $0,1,2$ or is infinite.   Clearly
$e(G)=0$ if and only if $G$ is finite.  It is well known that $e(G)=2$ if and
only if $G$ is virtually cyclic.  Stallings' famous Ends Theorem~\cite{Stal68,Stal71} says that
$e(G)\geq 2$ if and only if $G$ splits non-trivially as an amalgamated free
product or an HNN extension over a finite subgroup, or, equivalently, $G$ has
an edge-transitively action without inversions on a simplicial tree with no
global fixed points and finite edge stabilizers.

Our main theorem is then the following result.

\begin{theorem}\label{newmain}
Let $G$ be a finitely generated group with two or more ends.  Then the
submonoid membership and rational subset membership problems for $G$ are
recursively equivalent.  Moreover, there is a fixed rational subset of $G$
with undecidable membership problem if and only if there is a fixed
finitely generated submonoid of $G$ with
undecidable membership problem.
\end{theorem}

We remark that the proof of the last statement in Theorem~\ref{newmain}
is an existence argument:  we do not give an algorithm
that constructs a finitely generated submonoid with undecidable membership problem
from a rational subset with undecidable membership problem.

The case that $G$ in Theorem~\ref{newmain} has two ends is clear
since a two-ended group is virtually cyclic and hence has a decidable
rational subset membership problem.
So the interesting case is a group with infinitely many ends.

The paper is organized as follows.  Section~\ref {Sec acting on trees}
proves some results about groups acting on trees that
will allow us to reduce Theorem~\ref{newmain} to a
special case.  This special case is then handled in Section~\ref{Sec main lemma}.

\section{Groups acting on trees} \label{Sec acting on trees}

The goal of this section is to reduce our considerations to a particular type of HNN extension.
The notion of a (simplicial) tree is understood in the sense of Serre~\cite{Ser80}. An edge $e$ with initial vertex $v$ and
terminal vertex $w$ is written as $v\xrightarrow{\,e\,}w$.
Then, there is an inverse edge $w\xrightarrow{\,e^{-1}\,}v$, which never equals $e$.
The pair $\{e,e^{-1}\}$ is called a \emph{geometric edge}, as it corresponds to an edge of the geometric realization of the tree.  Often we do not distinguish between an edge and the corresponding geometric edge.
Let $T$ be a tree. If $v,w$ are vertices of a tree $T$, then $[v,w]$ denotes
the geodesic (or reduced) path  from $v$ to $w$ in $T$.
Note that a path in $T$ is geodesic if and only if it does not contain any backtracking, i.e, an edge
followed by its inverse edge.
An automorphism $g$ of $T$ is \emph{without inversions} if,
for each edge $e$ of $T$, $ge\neq e^{-1}$.
Let $G$ be a group acting on $T$ (by
automorphisms of $T$). Then $G$ acts \emph{without inversions on $T$}
if every $g \in G$ is without inversions.  Throughout this paper, we tacitly assume that all actions of groups on trees are without inversions.
We say that $G$ acts \emph{edge-transitively} on $T$
if $G$ acts transitively on the set of geometric edges.
The stabilizer in $G$ of a vertex or edge $x$ of $T$ will be denoted
$G_x$.  Of course, $G_e=G_{e\inv}$ for any edge $e$. A vertex $v$ is called a \emph{global fixed point} if $G = G_v$.
The following lemma is well known, but we include a proof for completeness.

\begin{lemma}\label{uniquemaximalfinite}
Let $G$ be a finitely generated group acting edge-transitively on a tree $T$
without global fixed points and with finite edge stabilizers (thus $G$ has more than
one end). Then the kernel of the action of $G$ on $T$ is the unique maximal finite normal subgroup of $G$.
\end{lemma}
\begin{proof}
Let $G_e$ be an edge stabilizer; it is finite by assumption.  The kernel $X$ of the action of $G$ on $T$ is clearly contained in $G_e$ and hence is a finite normal subgroup.  Let us show that every finite normal subgroup $N$ of $G$ is contained in $X$. Since $N$ is finite, it fixes some vertex $v$ of $T$~\cite[p.~36]{Ser80}.  Since $G$ has no global fixed
point, we can find an element $g\in G$ such that $gv\neq v$.  Let $[v,gv]$ be
the geodesic from $v$ to $gv$ in $T$.  Then since $N=gNg\inv$ fixes $v$ and $gv$, it
fixes $[v,gv]$ (since $T$ is a tree) and hence it fixes some edge of this
geodesic.  Since $G$ acts edge-transitively on $T$ and $N$ is normal,
it follows that $N$ fixes every edge of $T$ and hence $N\subseteq X$.  This completes the proof.
\qed
\end{proof}

The chief result in this section is a reduction to the case of an HNN extension of a very special sort.
Recall that if $H$ is a group and $\varphi\colon A\to B$ is a \emph{partial automorphism}
of $H$, i.e., an isomorphism between subgroups $A$ and $B$ of $H$, then the corresponding
HNN extension is the group $\bigast_{\varphi} H$ given by the presentation $\langle H,t\mid t\inv at=\varphi(a)\ (a\in A)\rangle$.

We first need  some results concerning groups acting on trees, which may be
of interest in their own right.  We recall a result of Tits (cf.~\cite[p.~63]{Ser80})
that an automorphism $g$ of a tree $T$ (without inversions) either
fixes a vertex, in which case it is called \emph{elliptic}, or it leaves
invariant a unique line $T_g$, called its \emph{axis}, on which it acts via
translation by some positive integer length.  Elements of the latter sort are
called \emph{hyperbolic}. Of course, hyperbolic elements have infinite order.

Our next lemma is a geometric version of
a result proved combinatorially (and separately) for amalgamated free products and HNN extensions with finite edge groups~\cite{LohSen05pos}.

\begin{lemma}\label{constructsequence}
Let $G$ be a finitely generated group acting edge-transitively on a tree $T$
without global fixed points and with finite edge stabilizers (thus $G$ has more than
one end). Let $v\xrightarrow{\,e\,}w$
be an edge of $T$ and suppose that $G_e$ is finite proper subgroup of both
$G_v$ and $G_w$.  Then
there is an element $s\in G$ such that $G_e\cap sG_es\inv=X$ where $X$ is the
largest finite normal subgroup of $G$.
\end{lemma}

\begin{proof}
Since the pointwise stabilizer of a geodesic is the
intersection of the stabilizers of all its edges (and hence has size bounded by
$|G_e|$), it follows that there is a non-empty finite geodesic path $p$ whose
pointwise stabilizer $H$ has smallest cardinality  amongst all geodesics in $T$.
Without loss of generality, we may assume that $p$ starts with either $e$ or $e^{-1}$ (else
translate it to do so). Moreover, if $p$ starts with $e^{-1}$ we
can replace $e$ by $e^{-1}$ and hence we may assume that $p$ starts with $e$.
Let $(se)^{\varepsilon}$ with $s\in G$ and $\varepsilon=\pm 1$ be the last edge traversed by
$p$.  We claim that we can choose $p$ so that $\varepsilon=1$.
If this is not the case, i.e., the last vertex of $p$
is $sv$, then we can choose $g\in G_{sv}\setminus G_{se}$
and consider the path $q=p(gse)$.  We then have the situation in
Figure~\ref{Fig:wrongway}.
\begin{figure}[tbhp]
\begin{center}
\setlength{\unitlength}{2mm}
\begin{picture}(47,10)
 \gasset{Nframe=n,Nfill=n,Nadjust=wh,Nadjustdist=.5,AHnb=1,ELdist=.5}
 \node(v)(0,5){$v$}
 \node(w)(10,5){$w$}
 \node(sw)(30,5){$sw$}
 \node(sv)(40,5){$sv$}
 \node(gsw)(47,0){$gsw$}
 \drawedge(v,w){$e$}
 \drawedge(sv,sw){$se$}
 \drawedge(sv,gsw){$gse$}
 \drawbpedge[dash={.5}0](w,-30,10,sw,150,10){$[w,sw]$}
 \end{picture}
\end{center}
\caption{When $sv$ is the last vertex of $p$\label{Fig:wrongway}.}
\end{figure}
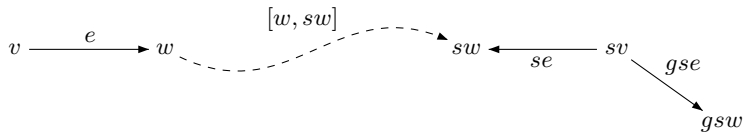
Any element of $G$ that fixes $q$ must then fix $p$ and hence the pointwise
stabilizer of $q$ coincides with $H$ by minimality.  Thus replacing $s$ by
$gs$, we are in the desired situation.   In particular, it follows that
we may take $s$ to be a hyperbolic element with
$[v,sv]$ a fundamental domain for the action of $s$ on its axis $T_s$, as the edges $e,se$ are coherent in the sense of~\cite[p.~62]{Ser80}.
Since every element that stabilizes $e$ and $se$ stabilizes $p$,
we have $H = G_e\cap sG_es\inv$. We will show that
$H=X$, where $X$ is the largest finite normal subgroup of $G$.
Recall from Lemma~\ref{uniquemaximalfinite} that $X$ is the kernel
of the action of $G$ on $T$. Hence, $X\subseteq H$. It remains to
show that $H \subseteq X$, i.e., that $H$ stabilizes every edge of $T$.

First we show that $H$ is the pointwise stabilizer of $T_s$.
Clearly, the pointwise stabilizer of $T_s$ is contained in $H$, since $p$ is contained in $T_s$.
If the pointwise stabilizer of $T_s$ is not equal to $H$, then some edge $f$
of $T_s$ is not stabilized by some element of $H$.  By extending $p$ to a reduced path in the line $T_s$
containing $f$, we obtain a reduced path with a smaller pointwise stabilizer
than $p$, a contradiction.

Next, observe that $H$ is closed
under conjugation by all powers of $s$ since $T_s$ is $\langle
s\rangle$-invariant and $H$ is the pointwise stabilizer of $T_s$.
Suppose now that $ge$ is an edge of $T\setminus T_s$ with $g\in G$. Let $r$ be the
geodesic path from $ge$ to $T_s$.  Clearly, $H$ stabilizes $ge$ if and only if
it stabilizes $s^nge$ for some $n$ since $s^nHs^{-n}=H$. Moreover, the
geodesic path from $s^nge$ to $T_s$ is $s^n r$. Thus without loss of
generality, we may assume that $r$ ends at a vertex $u$ of $s^m[v,sv]$ with
$m>2$.  Then  $p[sw,u]r\inv (ge)^{\varepsilon}$  is a reduced path
(see Figure~\ref{Fig:path-p}) containing $p$ as an initial segment, for an appropriate choice of $\varepsilon =\pm 1$,
and hence has pointwise stabilizer contained in $H$.  But by choice of $p$,
the pointwise stabilizer of $p[sw,u]r\inv (ge)^{\varepsilon}$
cannot be properly contained in $H$. Thus $H$ stabilizes $ge$,
i.e, $H$ stabilizes every edge of $T$ and hence $H=X$. This completes the proof.
\begin{figure}[tbhp]
\begin{center}
\setlength{\unitlength}{1.8mm}
\begin{picture}(62,24)(0,-8)
 \gasset{Nframe=n,Nfill=n,Nadjust=wh,Nadjustdist=.4,ELdist=.5,linewidth=0.12,AHnb=0}
 \node(c)(0,-5){}
 \node(v)(10,-5){$v$}
 \node(w)(18,-5){$w$}
 \drawedge(c,v){}
 \drawedge[AHnb=1](v,w){$e$}
 \node(sv)(30,-5){$sv$}
 \node(sw)(38,-5){$sw$}
 \drawedge[AHnb=1](sv,sw){$se$}
 \drawedge(w,sv){}
 \node(u)(50,-5){$u$}
 \drawedge(sw,u){}
 \node(d)(62,-5){}
 \drawedge(u,d){$T_s$}
 \node(a)(50,13){}
 \node[Nadjustdist=0](b)(50,5){}
 \drawedge[AHnb=1](b,a){$(ge)^{\varepsilon}$}
 \drawbpedge[AHnb=1,dash={.2}0](u,130,5,b,-50,5){$r^{-1}$}
 \put(10,-3){$\overbrace{\hspace{50.4mm}}^p$}
 \end{picture}
\end{center}
\caption{The path $p[sw,u]r\inv (ge)^{\varepsilon}$ \label{Fig:path-p}.}
\end{figure}
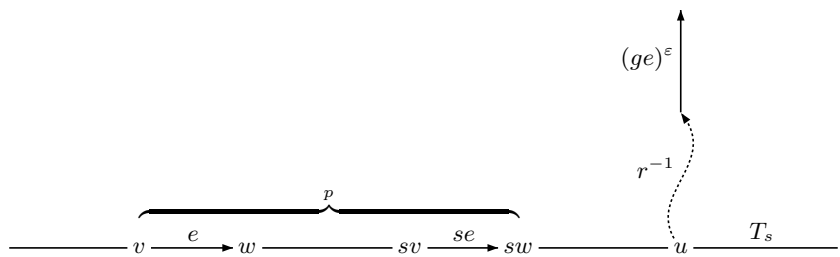
\qed
\end{proof}

Our next proposition uses the element $s$ constructed above to create subgroups of a special form.

\begin{proposition}\label{treeactionresult}
Let $G$ be a group acting on a tree $T$.   Let $v\xrightarrow{\, e\,}w$
be an edge of $T$ and put $N=\bigcap_{g\in G} gG_eg\inv$.  Suppose that:
\begin{enumerate}[(a)]
\item $[G_v:G_e]\geq 3$;
\item $[G_w:G_e]\geq 2$;
\item There is an element $s\in G$ with $sG_es\inv \cap G_e=N$.
\end{enumerate}
Then $G$ contains a subgroup isomorphic to an amalgamated free product of the
form $G_v\ast _{N} (N\rtimes \mathbb Z)$.  Moreover, if $G_e$ is finite, then
$G$ contains a subgroup isomorphic
to $G_v\ast _{N} (N\times \mathbb Z)\cong \bigast_{1_{N}} G_v$, where $1_N$ is the identity map on $N$.
\end{proposition}

\begin{proof}
We begin by observing that $N$ consists of those element of $G$ that stabilize all translates of $e$.
Our third hypothesis (c) then says that there is a translate $se$ so that the
subgroup of elements fixing both $e$ and $se$ is exactly $N$.
Our initial goal is to show that we may assume that the geodesic $[v,sv]$
from $v$ to $sv$ contains both $e$ and $se$.  First suppose the geodesic contains exactly
one of these edges.  Replacing $e$ by $se$ and $s$ by $s\inv$ if necessary,
we may assume without loss of generality that it contains $e$.
Since $G_{sw}\supsetneq G_{se}$ by (b), we can choose $g\in G_{sw}\setminus G_{se}$.
Then we have the situation in Figure~\ref{Fig:point1}.
\begin{figure}[tbhp]
\begin{center}
\setlength{\unitlength}{2mm}
\begin{picture}(47,10)
 \gasset{Nframe=n,Nfill=n,Nadjust=wh,Nadjustdist=.5,AHnb=1,ELdist=.5}
 \node(v)(0,5){$v$}
 \node(w)(10,5){$w$}
 \node(sv)(30,5){$sv$}
 \node(sw)(40,5){$sw$}
 \node(gsv)(47,0){$gsv$}
 \drawedge(v,w){$e$}
 \drawedge(sv,sw){$se$}
 \drawedge(gsv,sw){$gse$}
 \drawbpedge[dash={.5}0](w,-30,10,sv,150,10){$[w,sv]$}
 \end{picture}
\end{center}
\caption{When only $e$ is in $[v,sv]$\label{Fig:point1}.}
\end{figure}
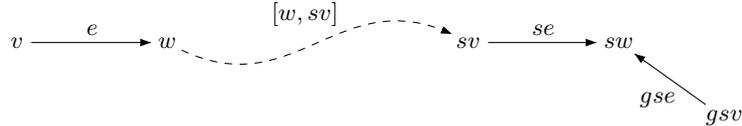
Any element of $G$ that fixes $e$ and $gse$ must then fix $e$ and $se$
and so $gsG_es\inv g\inv\cap G_e=N$.  Moreover, $e$ and $gse$ belong to $[v,gsv]$.
Thus replacing $s$ by $gs$, we are in the desired situation.

Next suppose that neither $e$ nor $se$ belong to
$[v,sv]$.  Choose $g\in G_{sw}\setminus G_{se}$.
Then we have the picture in Figure~\ref{Fig:point2}.
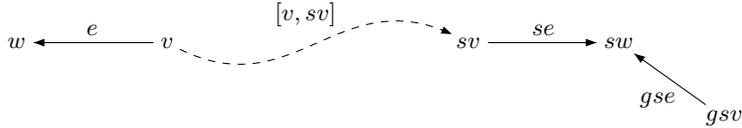
\begin{figure}[tbhp]
\begin{center}
\setlength{\unitlength}{2mm}
\begin{picture}(47,10)
 \gasset{Nframe=n,Nfill=n,Nadjust=wh,Nadjustdist=.5,AHnb=1,ELdist=.5}
 \node(w)(0,5){$w$}
 \node(v)(10,5){$v$}
 \node(sv)(30,5){$sv$}
 \node(sw)(40,5){$sw$}
 \node(gsv)(47,0){$gsv$}
 \drawedge[ELside=r](v,w){$e$}
 \drawedge(sv,sw){$se$}
 \drawedge(gsv,sw){$gse$}
 \drawbpedge[dash={.5}0](v,-30,10,sv,150,10){$[v,sv]$}
 \end{picture}
\end{center}
\caption{When neither $e$ nor $se$ is in $[v,sv]$\label{Fig:point2}.}
\end{figure}
Again any element of $G$ which fixes $e$ and $gse$ must fix also $se$ and hence
belong to $N$.  Thus  $gsG_es\inv g\inv\cap G_e=N$.
Moreover, $gse$ is an edge of $[v,gsv]$ and so replacing $s$
by $gs$ leads us to the previous case where exactly one of the edges is on the geodesic.

Thus we may now assume that $e$ and $se$ belong to $[v,sv]$.
Choose $g\in G_v\setminus G_e$.  Then we have the situation in
Figure~\ref{Fig:axisx}. The edges $e$ and $sge$ are coherent in the
sense of~\cite[p.~62]{Ser80}.
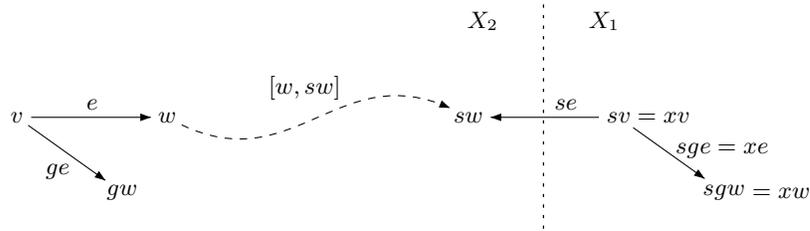
\begin{figure}[tbhp]
\begin{center}
\setlength{\unitlength}{2mm}
\begin{picture}(47,16)(0,-3)
 \gasset{Nframe=n,Nfill=n,Nadjust=wh,Nadjustdist=.5,AHnb=1,ELdist=.5}
 \node(v)(0,5){$v$}
 \node(w)(10,5){$w$}
 \node(sw)(30,5){$sw$}
 \node(sv)(40,5){$sv$}
 \node(sv')(43,5){$=xv$}
 \node(gw)(7,0){$gw$}
 \node(sgw)(47,0){$sgw$}
 \node(sgw')(50.8,0){$=xw$}
 \drawedge(v,w){$e$}
 \drawedge[ELpos=35,ELside=r](sv,sw){$se$}
 \drawedge[ELside=r](v,gw){$ge$}
 \drawedge[ELdist=0,ELpos=80](sv,sgw){$sge=xe$}
 \drawbpedge[dash={.5}0](w,-30,10,sw,150,10){$[w,sw]$}
 \node(a)(35,13){}
 \node(b)(35,-3){}
 \drawedge[ELdist=3,AHnb=0,ELpos=10,dash={0.2 0.5}0](a,b){$X_1$}
 \drawedge[ELdist=3,AHnb=0,ELpos=10,ELside=r,dash={0.2 0.5}0](a,b){$X_2$}
 \end{picture}
\end{center}
\caption{A part of the axis of $x=sg$\label{Fig:axisx}}
\end{figure}
Thus $x=sg$ is a hyperbolic element with axis $T_{x}=\langle x\rangle [v,sv]$
and $[v,sv]$ is a fundamental domain for the action of $x$ on $T_x$ (cf.~\cite[Sec.~I.6]{Ser80}).
Let $X_1$ be the connected component of $T\setminus \{se\}$
containing $sv$ and let $X_2$ be the connected component
of $T\setminus \{se\}$ containing $sw$.

Since $[G_v:G_e]\geq 3$ by (a) and $T_x$ contains only one other (geometric) edge incident on $v$
besides $e$, we can find $h\in G_v\setminus G_e$ so that $he\notin T_x$.  Then
$y=hxh\inv$ is a hyperbolic element with axis $T_y=hT_x$.  The axes $T_y$ and
$T_x$ intersect at the vertex $v=hv$, but do not coincide since $he\in
T_y\setminus T_x$.  Thus $T_y\cap T_x$ is either a point, a segment or an
infinite ray. Notice that $X_1\cap T_x$ and $X_2\cap T_x$ each contain exactly
one of the two ends of $T_x$.   Interchanging $e$ with $se$ and $s$ with
$s\inv$ in the statement of the proposition if necessary, we may assume that
$T_y\cap T_x$ does not contain the end of $T_x$ determined by $X_2\cap T_x$.
Then, since $x$ acts on $T_x$ by translations in the direction $v$ to $sv$, by
choosing $n>0$ large enough, we can guarantee that $x^nT_y\cap T_x$ is
contained in $X_1$ and does not contain the vertex $sv$.  Putting $u=x^nyx^{-n}$ yields a hyperbolic element $u\in
G$ with axis $T_u=x^nT_y$ such that $T_u\cap T_x\neq \emptyset$
is contained in $X_1$ and does not contain $sv$. The picture is either as in Figure~\ref{Fig:lastpic} or~\ref{Fig:lastpic2}.
\begin{figure}[tbhp]
\begin{center}
\setlength{\unitlength}{2mm}
\begin{picture}(54.5,20)(0,-5)
 \gasset{Nframe=n,Nfill=n,Nadjust=wh,Nadjustdist=.5,AHnb=1,ELdist=.5,linewidth=.1}
 \node(v)(0,5){$v$}
 \node(w)(8,5){$w$}
 \node(sw)(28,5){$sw$}
 \node(sv)(36,5){$sv=xv$}
 \drawedge(v,w){$e$}
 \drawedge[ELpos=65,ELside=r](sv,sw){$se$}
 \drawbpedge[dash={.5}0](w,-30,10,sw,140,10){$[w,sw]$}
 \gasset{Nfill=y,Nframe=y,Nadjust=wh,Nadjustdist=0,AHnb=0}
 \node(a)(32,15){}
 \node(b)(32,-5){}
 \drawedge[ELdist=3,ELpos=10,dash={0.2 0.5}0](a,b){$X_1$}
 \drawedge[ELdist=3,ELpos=10,ELside=r,dash={0.2 0.5}0](a,b){$X_2$}
 \drawcbezier(37.5,4,40,.07,42,.07,44,.07)
 \drawcbezier(37.5, -4, 40, -.07, 42, -.07, 44, -.07)
 \drawline(44,.07)(48,.07)
 \drawline(44,-.07)(48,-.07)
 \drawcbezier(48,.07, 50,.07,52,.07, 54.5,4)
 \drawcbezier(48, -.07, 50, -.07, 52, -.07, 54.5, -4)
 \put(40,1.7){$T_x$}\put(40,-2.5){$T_u$}
 \end{picture}
\end{center}
\caption{The axes $T_x$ and $T_u$ have bounded intersection\label{Fig:lastpic}.}
\end{figure}
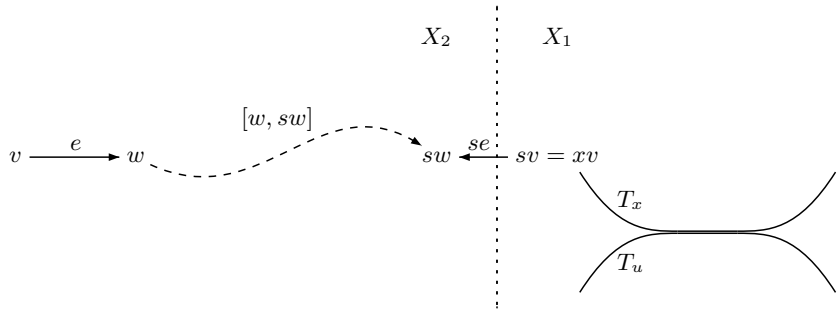
\begin{figure}[tbhp]
\begin{center}
\setlength{\unitlength}{2mm}
\begin{picture}(52,20)(0,-5)
 \gasset{Nframe=n,Nfill=n,Nadjust=wh,Nadjustdist=.5,AHnb=1,ELdist=.5,linewidth=.1}
 \node(v)(0,5){$v$}
 \node(w)(8,5){$w$}
 \node(sw)(28,5){$sw$}
 \node(sv)(36,5){$sv=xv$}
 \drawedge(v,w){$e$}
 \drawedge[ELpos=65,ELside=r](sv,sw){$se$}
 \drawbpedge[dash={.5}0](w,-30,10,sw,140,10){$[w,sw]$}
 \gasset{Nfill=y,Nframe=y,Nadjust=wh,Nadjustdist=0,AHnb=0}
 \node(a)(32,15){}
 \node(b)(32,-5){}
 \drawedge[ELdist=3,ELpos=10,dash={0.2 0.5}0](a,b){$X_1$}
 \drawedge[ELdist=3,ELpos=10,ELside=r,dash={0.2 0.5}0](a,b){$X_2$}
 \drawcbezier(37.5,4,40,.07,42,.07,44,.07)
 \drawcbezier(37.5, -4, 40, -.07, 42, -.07, 44, -.07)
 \drawline(44,.07)(52,.07)
 \drawline(44,-.07)(52,-.07)
 \put(40,1.7){$T_x$}\put(40,-2.5){$T_u$}
 \end{picture}
\end{center}
\caption{The axes $T_x$ and $T_u$ have unbounded intersection\label{Fig:lastpic2}.}
\end{figure}
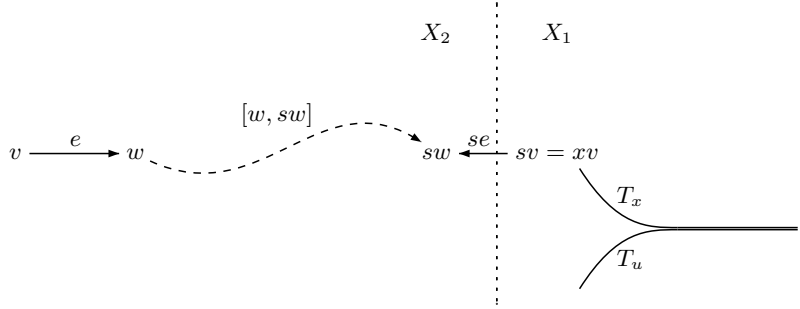

Let $H_1=\langle N,u\rangle$ and $H_2 = G_v$.  Since $u$ is hyperbolic and
elements of $N$ are elliptic, it follows that $\langle u\rangle\cap N=\{1\}$.
Consequently,  $H_1\cong N\rtimes \mathbb Z$ via the conjugation action of
$u$ on $N$.  Also, if $N$ is finite, then we can find a power $n>0$ so that
$u^n$ acts trivially on $N$ by conjugation.  Then $u^n$ is a hyperbolic
element with the same axis as $u$, but a larger translation length.
Replacing $u$ by $u^n$, we may then assume that $u$ commutes elementwise with
$N$ and so $H_1\cong N\times \mathbb Z$.  Also note that $H_1\cap
H_2=N$.  Indeed, if $h\in H_1$, we can write $h=u^na$ with $a\in N$.  Then
since $a\in G_v$, we have $h\in G_v$ if and only if $u^n\in G_v$, if and only
if $n=0$ since $u$ is hyperbolic.

We claim that $H_1$ and $H_2$ generate a subgroup isomorphic to $H_1 \ast_N
H_2$.  To prove this, we use the so-called Ping Pong Lemma~\cite{LySch77}.  The Ping Pong
Lemma has the following setup:   a group $G$ acting on a set $X$, subgroups
$H_1$ and $H_2$ of $G$ with $[H_1:N]\geq 3, [H_2:N]\geq 2$, where $N=H_1\cap H_2$, and non-empty subsets
$X_1,X_2\subseteq X$ with
$X_2\nsubseteq X_1$ such that:
\begin{itemize}
\item $h_1X_2\subseteq X_1$ for all $h_1\in H_1\setminus N$;
\item $h_2X_1\subseteq X_2$ for all $h_2\in H_2\setminus N$.
\end{itemize}
The conclusion of the lemma is that the subgroup of $G$ generated by $H_1$ and
$H_2$ is isomorphic to the amalgamated free product $H_1\ast_N H_2$.  We apply
the Ping Pong Lemma to our subgroups $H_1,H_2$.  We take $X$ to be $T$,
whereas $X_1$ and $X_2$ have already been defined above.  One should think of
the edge $se$ as the net of the ping pong table $X$.
Note that in our case $[H_1:N]=\infty$, whereas $[H_2:N]\geq [G_v:G_e]\geq 3$.

Let $k=u^ma\in H_1\setminus N$ with $a\in N$. Then $a$ fixes $[v,sv]$ because
it fixes $e$ and $se$. Let $q$ be the geodesic from
$sv$ to $T_u \cap T_x$; it is non-empty by choice of $u$.
As any non-trivial element $u^m$ of $\langle u\rangle$ acts
as a translation on the axis $T_u$, it follows that the geodesic
$p=[ksv,sv]=[u^msv,sv]$ must pass through a vertex of $T_u\cap T_x$ as per
Figure~\ref{Fig:geodesic}.
\begin{figure}[tbhp]
\begin{center}
\setlength{\unitlength}{2mm}
\begin{picture}(28,20)(24,-5)
 \gasset{Nframe=n,Nfill=n,Nadjust=wh,Nadjustdist=.4}
 \node(sw)(24,5){$sw$}
 \node(sv)(36,5){$sv$}
 \node(hsv)(46,5){$ksv$}
 \node(sgw')(50.2,5){$=u^msv$}
 \gasset{Nfill=y,Nframe=y,Nadjust=wh,Nadjustdist=0,AHnb=0,ELdist=.5,linewidth=.1}
 \node(a)(31,15){}
 \node(b)(31,-6){}
 \drawedge[ELdist=3,ELpos=10,dash={0.2 0.5}0](a,b){$X_1$}
 \drawedge[ELdist=3,ELpos=10,ELside=r,dash={0.2 0.5}0](a,b){$X_2$}
 \drawedge[ELpos=65,ELside=r,AHnb=1](sv,sw){$se$}
 \node(c)(33,-5){}
 \node(d)(52,-5){}
 \drawedge[ELside=r](c,d){$T_u$}
 \node(e)(36,-5){}
 \node(f)(46,-5){}
 \drawbpedge[dash={.2}0](sv,-50,5,e,130,5){$q$}
 \drawbpedge[dash={.2}0](hsv,-50,5,f,130,5){$u^mq$}
 \end{picture}
\end{center}
\caption{The geodesic $p=[ksv,sv] = [u^msv,sv]$\label{Fig:geodesic}}
\end{figure}
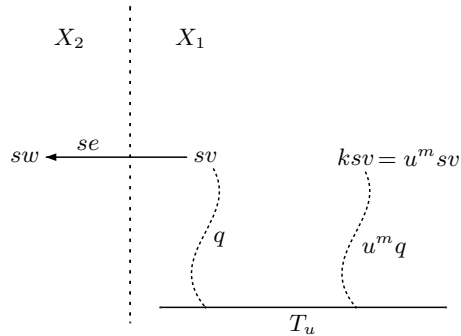
Since $T_u\cap T_x$ is contained in $X_1$, it follows that $p$ is contained in
$X_1$.  Let $v_0\in X_2$.  We claim that
$k[v_0,sv]\cup [ksv,sv] = [kv_0,ksv]\cup p$ is the geodesic from $kv_0$ to $sv$,
which implies that this path (and hence $kv_0$) is contained in $X_1$
since $p$ is contained in $X_1$.
To show that $k[v_0,sv]\cup [ksv,sv]$ is a geodesic, note that
$k[v_0,sv]$ ends with the edge $(kse)^{-1}$.
Hence, it suffices to show that $p$ does not begin with the edge
$kse$. The geodesic $p = [u^msv,sv]$ begins with
the path $u^mq$, see Figure~\ref{Fig:geodesic}.
Since $q \subseteq X_1$ does not contain $se$, $u^mq$ does not contain $kse=u^mse$.
We have thus shown $kX_2\subseteq X_1$.

On the other hand, if $k\in H_2\setminus N$, then $k$ stabilizes $v$ but not
$se$ (since if it stabilizes $v$ and $se$, then it must also stabilize $e$ and
so belong to $N$ by assumption on $s$). Since $kse\neq se$ and $kv=v$, there
is a last vertex $v_0$ on the geodesic $[v,sv]$ fixed by $k$  and $v_0\neq
sv$.  Let $v_1\in X_1$.  Then by definition of $X_1$, it follows that
$[v_1,v_0]=[v_1,sv]\cup [sv,v_0]$ and hence
$[kv_1,v_0]=k[v_1,v_0]=k[v_1,sv]\cup k[sv,v_0]$.  Notice that
$k[sv,v_0]=[ksv,v_0]$ cannot contain any edge of $[v,sv]$ by choice of $v_0$.
Therefore, $k[v_1,sv]\cup k[sv,v_0]\cup [v_0,sv]=[kv_1,sv]$ (see Figure~\ref{Fig:2nd-inclusion}) and hence
$[kv_1,sv]$ contains the edge $se$.  It follows that $kv_1\in X_2$ and so
$kX_1\subseteq X_2$.  The Ping Pong Lemma now
yields that  $\langle H_1,H_2\rangle\cong H_1\ast_N H_2$, completing the proof.
\qed
\begin{figure}[tbhp]
\begin{center}
\setlength{\unitlength}{2mm}
\begin{picture}(52,28)(0,-5)
 \gasset{Nframe=n,Nfill=n,Nadjust=wh,Nadjustdist=.5,AHnb=1,ELdist=.5,linewidth=.1}
 \node(v)(0,-3){$v$}
 \node(sv)(36,-3){$sv$}
 \node(v0)(18,-3){$v_0$}
 \node(ksv)(18,15){$ksv$}
 \node(v1)(42,3){$v_1$}
 \node(kv1)(24,21){$kv_1$}
 \drawbpedge(v,-30,10,v0,140,10){}
 \drawbpedge(v0,-30,10,sv,140,10){}
 \drawbpedge(v0,60,10,ksv,240,10){}
 \drawbpedge(sv,15,5,v1,195,5){}
 \drawbpedge(ksv,15,5,kv1,195,5){}
 \gasset{Nfill=y,Nframe=y,Nadjust=wh,Nadjustdist=0,AHnb=0}
 \node(a)(32,22){}
 \node(b)(32,-5){}
 \drawedge[ELdist=3,ELpos=20,dash={0.2 0.5}0](a,b){$X_1$}
 \drawedge[ELdist=3,ELpos=20,ELside=r,dash={0.2 0.5}0](a,b){$X_2$}
 \end{picture}
\end{center}
\caption{The inclusion $kX_1\subseteq X_2$.\label{Fig:2nd-inclusion}}
\end{figure}
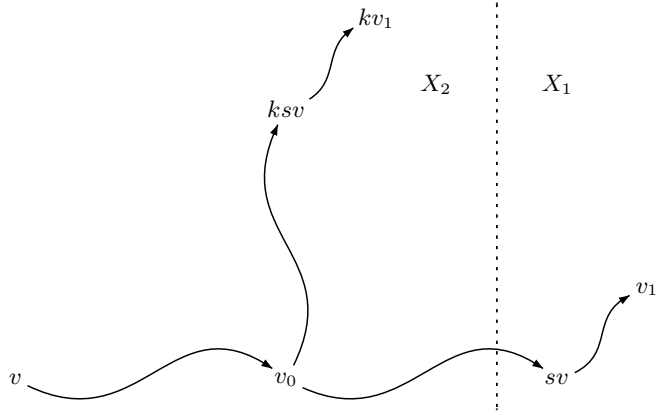
\end{proof}

The proposition admits the following corollary.

\begin{corollary}\label{reducetospecial}
Let $G$ be a finitely generated group splitting non-trivially over a finite subgroup $A$.
Suppose that $H$ is an infinite vertex group of the corresponding splitting.   Then $G$
has a subgroup isomorphic to the HNN extension
$\bigast_{1_X} H$ where $X$ is the largest finite normal subgroup of $G$ and
$1_X\colon X\to X$ is the identity mapping.
\end{corollary}
\begin{proof}
Let $T$ be the Bass-Serre tree associated to the splitting of $G$ over $A$.
It follows that there is an edge $v\xrightarrow{\, e\,}w$ of $T$ so
that $G_e=A$ and $G_v=H$.
Moreover, since the splitting is non-trivial and also
$G$ cannot be an ascending HNN extension of $H$ as $H$ is infinite and $A$ is finite, we must have $G_e \subsetneq G_w$.
The action of $G$ on $T$ is edge-transitively and has no global fixed point.
Lemma~\ref{constructsequence} provides $s\in S$ with \[s G_es\inv\cap G_e= sAs\inv \cap
A=X.\] Observing that $X=\bigcap_{g\in G}gG_eg\inv$, Proposition~\ref{treeactionresult} then allows us to conclude that $G$
contains a subgroup
isomorphic to $H\ast_X (X\times \mathbb Z)$, which is easily seen to be isomorphic to the desired HNN extension.
\qed
\end{proof}

The following technical result, which is the heart of this paper, will be proved in the next section.

\begin{lemma}\label{techlemma}
Let $G=\bigast_{1_X} H$ be an HNN extension with $X$ a finite proper normal
subgroup of $H$.  Then the rational subset membership problem for $H$ can be
reduced to the submonoid membership problem for $G$.  Moreover, if $H$ has a
fixed rational subset with undecidable membership problem, then $G$ has a
fixed finitely generated
submonoid with undecidable membership problem.
\end{lemma}

Let us now prove Theorem~\ref{newmain} assuming Lemma~\ref{techlemma}.

\subsubsection*{Proof of Theorem~\ref{newmain}}
For the first statement of Theorem~\ref{newmain}, the non-trivial direction is
to show that the rational subset membership problem for $G$ reduces to the
submonoid membership problem for $G$.   By Stallings' Ends Theorem~\cite{Stal68,Stal71} the group $G$ splits non-trivially over a finite subgroup $A$.
By the results of~\cite{KaSiSt06,LohSen08}, the rational subset
membership problem for $G$ reduces to the rational subset
membership problem for the vertex group(s) of the splitting.  Hence, it suffices to show that if $H$ is a vertex group of the splitting, then $H$ has
decidable rational subset membership problem. If $H$ is finite, there is
nothing to prove, so we may assume that $H$ is infinite.  Then by
Corollary~\ref{reducetospecial}, $G$ contains a subgroup isomorphic to
$\bigast_{1_X} H$. Clearly, also $\bigast_{1_X} H$ has a decidable submonoid
membership problem. The result now follows from Lemma~\ref{techlemma}.

The non-trivial direction of the second statement is to show that if $G$ has a
fixed rational subset with undecidable membership problem, then it has a fixed
finitely generated submonoid with undecidable membership problem.
In~\cite{LohSen08} the membership problem for a fixed rational subset of $G$
is reduced to the membership problem of a finite number of fixed rational
subsets of the vertex groups (in the splitting over $A$).  Hence one of the
vertex groups $H$ has a fixed rational subset with undecidable membership problem; necessarily $H$ is infinite.
Corollary~\ref{reducetospecial} then applies to
allow us to deduce that $G$ has a fixed subgroup of the form $\bigast_{1_X} H$
where $X$ is a finite proper normal subgroup of $H$ and $H$ has a fixed
rational subset with undecidable membership problem.  Thus by
Lemma~\ref{techlemma}, we conclude $G$ has a fixed finitely generated
submonoid with undecidable
membership problem.  This completes the proof.
\qed

\section{The Proof of Lemma~\ref{techlemma}} \label{Sec main lemma}

Let us fix an HNN-extension
$G = \langle H, t\mid  t^{-1} x t = x\ (x\in X) \rangle$,
where $X$ is a proper finite normal subgroup of the finitely generated group
$H$.  We assume furthermore
that $H$ is infinite.  The group $H$ acts on $X$ by conjugation and since $X$
is finite, the kernel $N$ of this action has finite index in $H$ (and hence
is finitely generated). In particular, $N$ is infinite and so $N\cap X$ is a proper subgroup of
$N$. Since decidability of rational subset membership is a virtual
property~\cite{Gru99}, to show that $H$ has decidable rational subset
membership, it suffices to prove that $N$ has decidable
rational subset membership.

Let $K=\langle N,t\rangle \leq G$.  Then $K$ centralizes $X$
and so in particular,  $K\cap X$ is contained in the
center of $K$.  Let us fix a finite group generating set $\Sigma$ for $N$
and denote by $\pi \colon  \widetilde{\Sigma}^* \to N$ the canonical projection.
Without loss of generality assume that $N \cap X \subseteq \widetilde{\Sigma}$.
Suppose that $\mathscr A = (Q, \widetilde{\Sigma}, \delta, q_0, F)$ is a finite automaton,
For a state $q \in Q$, let $X(q,\mathscr A)$ be the set of
all elements from $X\cap N$ represented by a word labeling a loop at state $q$. Note that
these elements form a submonoid of $X\cap N$ and therefore (since $X\cap N$ is finite)
a subgroup of $X\cap N$.
Let $X(\mathscr A) \subseteq \widetilde{\Sigma}$ be the set
$\bigcup_{q \in Q} X(q,\mathscr A)$.

The proof of the following lemma is quite similar to the proof of
Lemma~11 and Theorem~7 in \cite{LohSte08}.

\begin{lemma} \label{lemma main}
{}From a given finite automaton $\mathscr A$ over $\widetilde{\Sigma}$ and an element
$h \in N$ (given as a word over $\widetilde{\Sigma}$)
we can construct effectively  a finite subset $\Delta \subseteq K$
and an element $g \in K$ such that
$h \in \pi(X(\mathscr A)^*L(\mathscr A))$ if and only if $g \in \Delta^*$.
\end{lemma}

\begin{proof}
Let $\mathscr A = (Q, \widetilde{\Sigma}, \delta, q_0, F)$.
Without loss of generality assume that $Q \subseteq \mathbb{N}$ and
$q \geq 1$ for all $q \in Q$.
By introducing $\varepsilon$-transitions, we may also assume that the set
of final states $F$ consists of a single state $q_f$. We will
construct a finite subset $\Delta \subseteq K$
and an element $g \in K$ such that
$h \in \pi(X(\mathscr A)^*L(\mathscr A))$ if and only if $g \in \Delta^*$.

Fix an element $k \in N \setminus X$ (hence also $k\inv \in N \setminus X$).  Without loss of generality we may assume that $k\in \Sigma$.
For every $q \in Q \subseteq \mathbb{N}$, define $[q] \in K$ by
\[[q] = t^q k t^{-q}.\]
Note that $[q] x = x [q]$ for all $x \in X$ since $K$ centralizes $X$. Let
\begin{equation}\label{defineDelta}
\Delta = \{ [q] c [p]^{-1} \mid (q, c, p) \in \delta \} \quad\text{ and }\quad
g = [q_0] h [q_f]^{-1} .
\end{equation}
where we abuse notation by treating $\widetilde{\Sigma}$ as if it were a subset of $N$.
Observe that in \eqref{defineDelta}, we have $c \in \widetilde{\Sigma}
\cup \{ 1 \} \subseteq N$, since we introduced $\varepsilon$-transitions.
Also $\Delta\subseteq K$.  Note that $X(\mathscr A)^*$ is contained in
$\Delta^*$.  Indeed, if $x\in X(q,\mathscr A)$, then we can write $x=a_1\cdots
a_m$ where there are transitions \[(q,a_1,q_1),\ldots, (q_{m-1},a_m,q)\in
\delta.\]  Then, since $x$ commutes with $[q]$,
we have
\[x=[q]x[q]\inv = [q]a_1[q_1]\inv [q_1]a_2[q_2]\inv \cdots
[q_{m-1}]a_m[q]\inv\in \Delta^*.\]
It follows that $X(\mathscr A)^*\subseteq \Delta^*$.

We claim that $h \in \pi(X(\mathscr A)^*L(\mathscr A))$ if and only if $g \in \Delta^*$.
Let $\Gamma=\Sigma\cup \{t\}$.  Let us define a
\emph{$X$-cycle} to be word in $\widetilde \Gamma^*$ of the form
\[[q_1] h_1 [q_2]^{-1}  [q_2] h_{2} [q_3]^{-1}
\cdots [q_{\ell-1}] h_{\ell-1} [q_\ell]^{-1} [q_\ell] h_\ell [q_1]^{-1}\]
such that $\ell \geq 1$, $q_1, \ldots, q_\ell \in Q$,
$h_1, \ldots, h_\ell \in \widetilde \Sigma^*$, and
$h_1 \cdots h_\ell$ represents an element from the subgroup $X$.
Since all elements from $X$ commute with all $[q]$ ($q \in Q$),
an $X$-cycle equals an element from $X$ in $G$ (actually an element of $N\cap X$).

A word in $\widetilde \Gamma^*$ of the form
\[[q_1] h_1 [p_1]^{-1} \cdots[q_m] h_m [p_m]^{-1},\]
where $q_1, p_1, \ldots, q_m,p_m \in Q$ and
$h_1, \ldots, h_m \in \widetilde \Sigma^*$, is called \emph{$X$-cycle-free}
if it does not contain an $X$-cycle as a factor.

\medskip
\noindent
{\em Claim 1.}\label{newclaim1}
Suppose that $p_1,\ldots, p_m,q_1,\ldots,q_m\in Q$ with $p_i\neq q_i$
($1 \leq i \leq m$) and $q_i \neq p_{i+1}$ ($1 \leq i < m$).
Then the element
\begin{equation}\label{E Claim 1 new}
g=[p_1]\inv [q_1][p_2]\inv [q_2]\cdots [p_m]\inv [q_m]
\end{equation}
has a reduced expression in the HNN extension $G$ starting with $t$ and ending with $t\inv$.

\medskip
\noindent {\em Proof of Claim 1.}
We have
\begin{eqnarray*}
\prod_{i=1}^{m} [p_i]^{-1} [q_i]
 & = & \prod_{i=1}^{m} t^{p_i} k\inv t^{-p_i} t^{q_i} k
 t^{-q_{i}}  \nonumber \\
& = & \prod_{i=1}^{m} t^{p_i} k\inv t^{q_i-p_i} k
 t^{-q_{i}}  \nonumber \\
& = & t^{p_1} \biggl(\prod_{i=1}^{m-1} k\inv t^{q_i-p_i} k
 t^{p_{i+1}-q_{i}} \biggr) k\inv t^{q_m-p_{m}} k t^{-q_m} .
\end{eqnarray*}
The latter word is a reduced word for $g$ of the required form. This establishes Claim 1.

\medskip

\noindent
Let $Y$ be the set of elements of the form \eqref{E Claim 1 new}.  It follows
immediately from the claim that any element $g'$ of the form
$h_0y_1h_1y_2h_2\cdots y_rh_r$ with $y_1, \ldots, y_r \in Y$,
$h_1,\ldots, h_{r-1}\in N\setminus X$, and $h_0,h_r\in (N\setminus
X)\cup \{1\}$ has a reduced expression for the HNN extension $G$ containing
$t$ and $t\inv$ and hence does not belong to $N$.
Such elements $g'$ will be called \emph{good}.  I.e.,
$g'$ is good if it can be written as an alternating word in elements of $Y$ and $N\setminus X$ with at least one factor from $Y$.

\medskip

\noindent
{\em Claim 2.}
Assume that in the HNN-extension $G$
\begin{equation}\label{E Claim 2}
[n] h [r]^{-1} = [q_1] h_1 [p_1]^{-1} \cdots[q_m] h_m [p_m]^{-1}
\end{equation}
where $n,r,q_1, p_1, \ldots, q_m,p_m \in Q$,
$h,h_1, \ldots, h_m \in \widetilde \Sigma^*$.
If $\prod_{i=1}^m [q_i] h_i [p_i]^{-1}$
is $X$-cycle-free, then either
$m=0$, $n=r$, and $h=1$, or all of the following hold:
\begin{itemize}
\item $m \geq 1$ and $p_i = q_{i+1}$ for all $1 \leq i < m$,
\item $q_1 =n$ and $p_m = r$, and
\item $h = h_1 \cdots h_m$ in $N$.
\end{itemize}

\medskip

\noindent {\em Proof of Claim 2.} We prove Claim 2 by induction over
$m$. If $m=0$, then \eqref{E Claim 2} becomes
\begin{equation*}
\label{L 1}
h=[n]\inv [r].
\end{equation*}
Since $h\in N$, Claim 1 immediately yields $n=r$, and hence $h=1$, as required.

Now assume that $m  \geq 1$.

\medskip

\noindent
{\em Case 1.} $m \geq 2$ and there is $1 \leq \ell < m$ such that $p_\ell = q_{\ell+1}$.
Then \eqref{E Claim 2} implies
\begin{equation*}
[n] h [r]^{-1} = \biggl( \prod_{i=1}^{\ell-1} [q_i] h_i [p_i]^{-1} \biggr)
[q_\ell] (h_\ell h_{\ell+1})[p_{\ell+1}]^{-1} \biggl( \prod_{i=\ell+2}^m [q_i] h_i [p_i]^{-1}\biggr)
\end{equation*}
in $G$.
Since $\prod_{i=1}^m [q_i] h_i [p_i]^{-1}$ from \eqref{E Claim 2} is $X$-cycle-free, also
the right hand side of this equality is $X$-cycle-free. Hence, we can apply the induction
hypothesis and obtain:
\begin{itemize}
\item $p_i = q_{i+1}$ for all $1 \leq i \leq \ell-1$ and all $\ell+1 \leq i < m$,
\item $q_1 =n$ and $p_m = r$, and
\item $h = h_1 \cdots h_m$ in $N$.
\end{itemize}
Since $p_\ell = q_{\ell+1}$ we obtain  $p_i = q_{i+1}$ for all $1 \leq i < m$.

\medskip

\noindent
{\em Case 2.} $p_i \neq q_{i+1}$ for all $1 \leq i < m$.
We have to show that $m=1$, $q_1 =n$, $p_m = r$, and $h=h_1$ in $N$.
Choose a set $T$ of coset representatives for $N/(N\cap X)$ with $1\in T$ and write $h_i=x_ih'_i$
with $h'_i\in T$ and $x_i\in N\cap X$, for $1 \leq i \leq m$.
{}From this definition we get $h_1 h_2 \cdots h_m = x h'_1 h'_2 \cdots h'_m$
where $x=x_1\cdots x_m\in N\cap X$ (since $N$ centralizes $X$). Moreover, we also have
\[\prod_{i=1}^m [q_i] h_i [p_i]^{-1} = x \prod_{i=1}^m [q_i] h'_i [p_i]^{-1}\]
since $N\cap X$ is central in $K$.

Thus, again using that $x$ is central in $K$, we have
\begin{equation}\label{newproofcase2eq1}
x\inv h= [n]\inv \left(\prod_{i=1}^m[q_i]h_i'[p_i]\inv\right)[r].
\end{equation}
Since the right hand side of \eqref{E Claim 2} is $X$-cycle-free, if $h_i'=1$
(i.e., $h_i\in X$), then $q_i\neq p_i$, for all $1 \leq i \leq m$ (else
$[q_i]h_i[p_i]\inv$ is an $X$-cycle).  Also we are assuming $p_i\neq q_{i+1}$.
If $h_i'\neq 1$, then $h_i'\in N\setminus X$.  Therefore, if both $n\neq q_1$
and $r\neq p_m$, then the right hand side of \eqref{newproofcase2eq1}
represents a good element of $G$ and so cannot belong to $N$, as was
observed just before the statement of Claim 2.  This contradicts $x\inv h\in
N$.  Similarly, suppose $n=q_1$ and $p_m\neq r$.  Then \[x\inv h=
h_1'[p_1]\inv\left(\prod_{i=2}^m[q_i]h_i'[p_i]\inv\right)[r]\] which is again
good, a contradiction.  The case $n\neq q_1$ and $p_m=r$, is handled
analogously.  If $n=q_1,p_m=r$ and $m\geq 2$, then  \[x\inv h=
h_1'[p_1]\inv\left(\prod_{i=2}^{m-1}[q_i]h_i'[p_i]\inv\right)[q_m]h_m'\]
and so is again good, a contradiction.
The only remaining case is when $n=p_1$, $r=p_m$ and $m=1$.
Then $[n]h[r]\inv =[n]h_1[r]\inv$ in $G$ and so $h=h_1$ in $N$, as required.  This proves Claim 2.
An immediate consequence of Claim 2 is:

\medskip

\noindent
{\em Claim 3.}
Assume that in the HNN-extension $G$
$$
[q_0] h [q_f]^{-1} = [q_1] a_1 [p_1]^{-1} \cdots[q_m] a_m [p_m]^{-1},
$$
where $(q_i, a_i, p_i) \in \delta$ (i.e., $[q_i] a_i [p_i]^{-1} \in \Delta$)
for $1 \leq i \leq m$. If $\prod_{i=1}^m [q_i] a_i [p_i]^{-1}$
is $X$-cycle-free, then $h \in \pi(L(\mathscr A))$.

\medskip

\noindent
The case $m \geq 1$ follows immediately from Claim~2. In case $m=0$,
Claim~2 implies $q_0 = q_f$ and $h=1$. But $q_0=q_f$ implies
$\varepsilon \in L(\mathscr A)$ and hence $h \in \pi(L(\mathscr A))$.

\medskip

\noindent
Now we are ready to prove that
$h \in \pi(X(\mathscr A)^*L(\mathscr A))$
if and only if $g = [q_0] h [q_f]^{-1} \in \Delta^*$.
First assume that $h \in \pi(X(\mathscr A)^*L(\mathscr A))$.
Let $h = xa_1 \cdots a_m$ in $N$, where
$a_1 \cdots a_m \in L(\mathscr A)$ and
$x\in X(\mathscr A)^*$.
Hence, there are states $q_0, \ldots, q_m$ with
$(q_{i-1}, a_i, q_i) \in \delta$ for $1 \leq i \leq m$,
and $q_m = q_f$.
It was observed earlier that $X(\mathscr A)^*\subseteq \Delta^*$.
Moreover, since $x$ commutes with $[q_0]$, we obtain in $G$:
\begin{equation*}
g=[q_0] h [q_f]^{-1} = x[q_0]a_1a_2 \cdots a_m[q_f]^{-1}
= x\prod_{i=1}^m [q_{i-1}] a_i [q_i]^{-1}\in \Delta^*.
\end{equation*}
Next assume that $g=[q_0] h [q_f]^{-1} \in \Delta^*$.
Thus,
\[[q_0] h [q_f]^{-1} =  [q_1] a_1 [p_1]^{-1} \cdots [q_m] a_m [p_m]^{-1}\]
in $G$,
where $q_1, p_1, \ldots, q_m,p_m \in Q$, $a_1,
\ldots, a_m \in \widetilde{\Sigma} \cup \{\varepsilon\}$, and
$(q_i,a_i, p_i) \in \delta$ for $1 \leq i \leq m$.
Every $X$-cycle that occurs in $\prod_{i=1}^m [q_i] a_i [p_i]^{-1}$
can be reduced in $G$ to an element from some subgroup of the form
$X(q,\mathscr A)$.  Moreover, since $X$ is central in $K$, we can move each of
these elements to the beginning and then to the left-hand side
$[q_0] h [q_f]^{-1}$ by taking inverses.
Performing this reduction until no further $X$-cycles remain, we obtain in
$G$ an equality
\begin{equation}
\label{X-cycle free sequence}
[q_0] x^{-1} h [q_f]^{-1} =  [n_1] b_1 [r_1]^{-1} [n_2] b_2 [r_2]^{-1}
             \cdots [n_\ell] b_\ell [r_\ell]^{-1}
\end{equation}
where $(n_i, b_i, r_i) \in \delta$ and $x\in X(\mathscr A)^*$.
Since the right-hand side of \eqref{X-cycle free sequence} is
$X$-cycle free, Claim~3 implies $x^{-1} h \in \pi(L(\mathscr A))$, i.e.,
$h \in \pi(X(\mathscr A)^*L(\mathscr A))$.
\qed
\end{proof}

Let us consider again a finite automaton $\mathscr A = (Q, \widetilde{\Sigma}, \delta, q_0, F)$ over the alphabet
$\widetilde{\Sigma}$. A subset $P \subseteq Q$ is called {\em admissible} if
$q_0 \in P$ and $P \cap F \neq \emptyset$.
For every admissible subset $P \subseteq Q$ we define the automaton $\mathscr A_P$ as follows:
\[\mathscr A_P = (P \times \{ R \mid R \subseteq P, q_0 \in R \} , \widetilde{\Sigma},
\delta_P, (q_0, \{q_0\}), \{ (q,P) \mid q \in P \cap F\} ),\]
where the transition relation $\delta_P$ is given by
\[\delta_P = \{ ( (p,R), a, (q, R\cup\{q \})) \mid p,q \in P, (p,a,q) \in
\delta\}.\]  The automaton $\mathscr A_P$ works as follows.  The construction
is such that there is a run in $\mathscr A_P$ of a word $w$ from
$(q_0,\{q_0\})$ to $(p,R)$ if and only if there is a run from $q_0$ to $p$ in
$\mathscr A$ that visits precisely the states in $R$.  In particular,
$L(\mathscr A_P)$ consists of all words $w\in L(\mathscr A)$ that label an
accepting path in $\mathscr A$ using exactly the vertices from $P$.

\begin{lemma} \label{lemma visit everything}
For every finite automaton $\mathscr A= (Q, \widetilde{\Sigma}, \delta, q_0, F)$
and $h \in N$, we have:
$h \in \pi(L(\mathscr A)) \ \Longleftrightarrow \ \exists P \subseteq Q\ \text{admissible such that}\ h \in \pi(L(\mathscr A_P))$.
\end{lemma}
\begin{proof}
Since $L(\mathscr A_P)\subseteq L(\mathscr A)$ for all admissible subsets, to
prove the lemma it suffices to observe that if $h=\pi(w)$ with $w\in
L(\mathscr A)$, then the set $P$ of states visited
by a successful run of $w$ is admissible and $w\in L(\mathscr A_P)$.\qed
\end{proof}



We are now ready to prove Lemma~\ref{techlemma}.

\subsubsection*{Proof of Lemma~\ref{techlemma}}
Assume that $G=\bigast_{1_X} H$, where $X$ is a proper finite normal subgroup of
$H$, has a decidable submonoid membership problem.  We shall exhibit an
algorithm for the rational subset membership problem for $H$.  If $H$ is
finite, there is nothing to prove.  So assume that $H$ is infinite.  Let $N$
be the centralizer of $X$ in $H$ and let $K=\langle N,t\rangle \leq G$.
Then $N$ has finite index in $H$ and so it suffices to show that $N$ has a decidable rational subset membership problem by~\cite{Gru99}.
For this, let $\mathscr A = (Q, \widetilde{\Sigma}, \delta, q_0, F)$ be a finite automaton and
let $h \in N$. We show how to decide $h \in L(\mathscr A)$.
By Lemma~\ref{lemma visit everything} it suffices
to decide whether $h \in \pi(L(\mathscr A_P))$ for some admissible subset $P \subseteq
Q$. Let us abbreviate the automaton $\mathscr A_P$ by $\mathscr B$.
Recall that the set of states of $\mathscr B$ is $P \times \{ R \subseteq P \mid
q_0 \in R \}$.

Now we apply Lemma~\ref{lemma main} and construct an element $g \in K$ and
a finite subset $\Delta \subseteq K$ such that
\[h \in \pi(X(\mathscr B)^*L(\mathscr B)) \quad
\Longleftrightarrow \quad g \in \Delta^*.\]
Hence it suffices to prove
$\pi(L(\mathscr B)) = \pi(X(\mathscr B)^*L(\mathscr B))$.

It is immediate that $\pi(L(\mathscr B)) \subseteq \pi(X(\mathscr B)^*L(\mathscr B))$. So it remains to establish that
$\pi(X(\mathscr B)^*L(\mathscr B))
\subseteq \pi(L(\mathscr B))$.  To do this, it suffices to show that
\[\pi(x L(\mathscr B))\subseteq \pi(L(\mathscr B))\] for every
$x \in X(\mathscr B)$ (then we can conclude by induction over the length
of the word from $X(\mathscr B)^*$).
Let us take $x  \in X((p,R),\mathscr B)$ for some
state $(p,R)$ of $\mathscr B$ and consider
a word $x a_1a_2 \cdots a_m$ with
$a_1 a_2 \cdots a_m \in L(\mathscr B)$.
Hence, there exist states $(q_0, R_0), \ldots, (q_m,R_m)$
of $\mathscr B$ such that:
\begin{itemize}
\item $R_0 = \{q_0\}$,
\item $((q_i,R_i),a_{i+1}, (q_{i+1},R_{i+1}))$ is a transition of
$\mathscr B$ for $0 \leq i \leq m-1$, and
\item $R_m = P$, $q_m \in P \cap F$.
\end{itemize}
Note that we have $P = \{ q_0, \ldots, q_m\}$.
Hence, there exists an $j$ such that $q_j = p$
(recall that $p\in P$ was such that $x \in  X((p,R),\mathscr B)$).
Since $X((p,R),\mathscr B)\subseteq N\cap X$ is central in $N$, we have
\[x a_1 a_2 \cdots a_m = a_1 \cdots a_j x a_{j+1} \cdots a_m\]
in $N$. Then there is
a loop at state $(p,R) = (q_j,R)$ labeled with
$c_1 \cdots c_k$ and $x = \pi(c_1 \cdots c_k)$ by definition of $X((p,R),\mathscr B)$.  In particular,
$c_1\cdots c_k$ reads a loop in $\mathscr A$ at $p$ visiting only states
contained in $R \subseteq P$.
Now, starting from the initial state $(q_0, R_0)$ we can read
the word $a_1 \cdots a_j c_1 \cdots c_k a_{j+1} \cdots a_m$ on the automaton
$\mathscr B$. The state reached after reading $a_1 \cdots a_j c_1 \cdots c_k a_{j+1}
\cdots a_\ell$, for $\ell \geq j$, is of the form
$(q_\ell, S_\ell)$, where $R_\ell \subseteq S_\ell$.
Hence, $S_m = P$. This shows that $a_1 \cdots a_j c_1 \cdots c_k
a_{j+1} \cdots a_m \in L(\mathscr B)$.  This completes the proof
that $N$ has decidable rational subset membership problem.
Therefore, $G$ has decidable rational subset membership problem.

On the other hand, if $H$ has a fixed rational subset with undecidable
membership problem, then it follows from the construction in~\cite{Gru99} that
$N$ has a fixed rational subset $\pi(L(\mathscr A))$ with undecidable
membership problem.  It follows from Lemma~\ref{lemma visit everything} that
$L(\mathscr A_P)$ has an undecidable membership
problem for some $P$.  Consequently, Lemma~\ref{lemma main}
provides a fixed finitely generated submonoid of $G$ with undecidable membership problem.

This completes the proof of Lemma~\ref{techlemma}, thereby establishing Theorem~\ref{newmain}.
\qed

\section{Concluding remarks and open problems}

We have shown that, for every group with at least two ends, the rational subset membership
problem and the submonoid membership problem are recursively equivalent.
Moreover, in a previous paper we proved that these two problems are recursively equivalent for graph
groups as well~\cite{LohSte08}.

It is easy to see that there exists a group $G$ with infinitely many ends for
which submonoid membership is undecidable but the generalized word problem
(i.e., membership in finitely generated subgroups) is decidable.
Take any group $H$ for which the generalized word problem is decidable
but submonoid membership is decidable (e.g., the free metabelian group
of rank $2$~\cite{LohSt09} or the graph group defined by a path with $4$ nodes~\cite{LohSte08}).
Then the same is true for $H\ast \mathbb{Z}$, which has infinitely many ends.

The obvious remaining open question is whether there exists a (necessarily one-ended)
finitely generated group for which rational subset membership is undecidable but submonoid
membership is decidable. We conjecture that such a group exists. We also
conjecture that decidability of submonoid membership is \emph{not} preserved
by free products, because otherwise rational subset membership and
submonoid membership would be equivalent for all finitely generated groups.

\bibliographystyle{abbrv}
\def\cprime{$'$}

\end{document}